\documentclass[11pt,twoside]{article}

\usepackage{mathrsfs,amsfonts,amsmath,amssymb}
\usepackage{latexsym}
\newtheorem{satz}{Theorem}
\newtheorem{proposition}[satz]{Proposition}
\newtheorem{theorem}{Theorem}
\newtheorem{lemma}[satz]{Lemma}

\newtheorem{corollary}[satz]{Corollary}
\newtheorem{remark}[satz]{Remark}

\def\T{\mathsf{T}}

\def\Z{\mathbb {Z}}
\def\F{\mathbb {F}}
\def\R{\mathbb {R}}
\def\E{\mathsf{E}}

\def\a{\alpha}

\def\C{\mathbb{C}}

\def\({\big (}
\def\){\big )}

\def\G{\Gamma}

\def\le{\leqslant}
\def\ge{\geqslant}
\def\_phi{\varphi}

\def\Gr{{\mathbf G}}

\def\ov{\overline}
\def\Spec{{\rm Spec\,}}
\renewcommand{\thetheorem}{\arabic{theorem}.}

\def\D{\Delta}

\def\oT{{\rm T}}
\def\tr{{\mathfrak{tr}}}
\newcommand{\ve}{\boldsymbol v}

\title{
On the few products, many sums problem}

\author{Brendan Murphy, Misha Rudnev, Ilya D. Shkredov and Yurii N. Shteinikov}
\begin{document}
\maketitle
\begin{abstract} We prove new results on additive properties of finite sets $A$ with small multiplicative doubling $|AA|\leq M|A|$ in the category of real/complex sets as well as multiplicative subgroups in the prime residue field. The improvements are based on new combinatorial lemmata, which may be of independent interest.

Our main results are the inequality 
$$
|A-A|^3|AA|^5 \gtrsim |A|^{10},
$$
over the reals, ``redistributing'' the exponents in the textbook Elekes sum-product inequality
and the new best known additive energy bound
$\E(A)\lesssim_M |A|^{49/20}$, which aligns, in a sense to be discussed, with the best known sum set bound $|A+A|\gtrsim_M |A|^{8/5}$.

These bounds, with $M=1$, also apply to multiplicative subgroups of $\F^\times_p$, whose order is $O(\sqrt{p})$. We adapt the above energy bound to larger subgroups and
obtain new bounds on gaps between elements in cosets of subgroups of order $\Omega(\sqrt{p})$.
\end{abstract}

\section{Introduction}  Let $A$ be a finite set in a field. We use the standard notation $A\pm A,\,AA,\,A/A $  for the sets of all sums, differences, products and finite ratios of pairs of elements of $A$, as well as $A^{-1}$ for the set of inverses, $A+a=A+\{a\}$ for translates, etc. By $A^k$, however, we mean the $k$-fold Cartesian project of $A$ with itself.

The Erd\H os-Szemer\'edi \cite{ES} or sum-product conjecture applied to reals challenges one to prove that $\forall \epsilon>0,$
\begin{equation}\label{ers}
|AA|+|A+ A| \geq |A|^{2-\epsilon},
\end{equation}
for all sufficiently large $A\subset \R$.

The {\em weak Erd\H os-Szemer\'edi conjecture}, or {\em few products, many sums} is a claim that
if $A$ has small multiplicative doubling, that is $|AA| \leq M|A|$ for some $M\geq 1$, then $|A\pm A|\gtrsim_M |A|^{2}$, where the inequality symbols $\gtrsim_M,\lesssim_M$  will subsume universal constants, powers of $\log|A|$  (logarithms are meant to be  base $2$) and powers of $M$ if the subscript ${\,}_M$ is present; constants alone are suppressed by the standard Vinogradov notation $\ll,\gg$ and, respectively $O,\Omega$ (as well as $\approx$ for both $O$ and $\Omega$); these can also be subscripted by ${\,}_M$ to hide powers of $M$. When both $\lesssim$ and $\gtrsim$ bounds hold we may use the symbol $\sim$. To ensure not dividing by zero, it is assumed by default in all formulations that $|A|>1$ as well as $0\not\in A$. 

In this paper we address the case $A \subset \R$ as well as when $A$ is a multiplicative subgroup of the multiplicative group $\F^\times_p$ of the prime residue field $\F_p$. All our results over the reals apply to the complex field as well: we do not make a distinction and {\em real} may be read in the sequel as {\em real or complex}.

The weak Erd\H os-Szemer\'edi conjecture appears to be the key issue in understanding the more general sum-product phenomenon, see e.g. a survey \cite{GS} and the references therein. If in its above formulation one allows exponential dependence on $M$,  the affirmative answer was established by  Chang and Solymosi \cite{ChS} via a variant of Schmidt's subspace theorem.

The strongest few products, many sums result known so far is due to by Bourgain and Chang \cite{BCh} in the context of integers (rationals), the proof relying strongly on the main theorem of arithmetics. It claims that for any $\epsilon>0$, there is a power $C(\epsilon)$, so that $|A\pm A|\gg M^{-C(\epsilon)}|A|^{2 - \epsilon},$ although $C(\epsilon)$ goes to infinity as $\epsilon\to 0$.

Our results are somewhat different in flavour, for our dependence $C(\epsilon)$ is linear (as well as all the constants involved are ``reasonable'' and computable); what we cannot do is go below $\epsilon= {1/3}$.

\medskip
The converse question {\em few sums, many products} is resolved over the reals, where it was shown by Elekes and Ruzsa  \cite{ER} to follow from the Szemer\'edi-Trotter theorem. Its strongest quantitative version is implied by Solymosi's \cite{So} inequality $|AA||A+A|^2\gtrsim |A|^4$, although this does not embrace $A-A$. Moreover, \cite[Theorem 12]{RSS} presents an affirmative quantitative estimate to an $L^2$-variant of the question. The few sums, many products question is not settled in positive characteristic, the best known results being  \cite[Theorem 2]{57}.

\medskip
One can target various type of estimates along the lines of the few products, many sums  -- henceforth FPMS -- phenomenon (the acronym is meant to embrace differences and ratios as well) in terms of the properties of the number of realisations function $r_{A\pm A}(x)$, namely
$$r_{A\pm A}(x) = : |\{ (a,b)\in A\times A:\,a\pm b = x\}|$$ and its moments, in particular
$$
\E(A) = : \sum_x r^2_{A\pm A}(x),
$$
known as (additive) energy. Energy is independent of the choice of $\pm$, being the number of solutions of the equation $a_1+a_2=a_3+a_4,$ with variables in $A$, which can be rearranged. A fruitful viewpoint at looking at $r_{A- A}(x)$ is that $x$ represents an equivalence class on $A\times A$ by translation, with $r_{A- A}(x)$ members.

\medskip
The three types of FPMS bounds one may be interested in are as follows.
\begin{itemize}
\item {\sf Convolution support:}  inequalities $|A\pm A|\gg_M |A|^{2-\epsilon}$ aiming at $\epsilon \to 0_+$.
\item {\sf Energy, or $L^2$-bounds:} inequalities $\E(A)\ll_M |A|^{2+\epsilon}$ aiming at $\epsilon \to 0_+$.
\item {\sf $L^\infty$-bounds:} inequalities $r_{A+A}(x) \ll_M |A|^{\epsilon}$, aiming at $\epsilon \to 0_+$, same for $r_{A-A}(x)$ for $x\neq 0$.
\end{itemize}
As far as the last question is concerned, there is a bound $O(|A|^{2/3})$ implied by a single application of the Szemer\'edi-Trotter theorem; nothing better appears to be known, and we therefore do not press the issue any further.

So this paper addresses the first two bullet points. Techniques available today have limited powers, and those in this paper prefer differences to sums, owing to shift-invariance.  


Energy bounds clearly imply ones for support: by the Cauchy-Schwarz inequality
$$
|A\pm A|\geq \frac{|A|^4}{\E(A)}.$$
Energy bounds are harder to establish, and in order to improve the threshold exponent $5/2$, which applies in a very wide context \cite{RRS}, the third author set forth in \cite{S_sp} an eigenvalue technique, which was then further developed in \cite{SS1},  \cite{S_ineq}. The corresponding threshold exponent $3/2$ for the support size can be improved without it, see, e.g. \cite{Li_R-N} as well as the proofs of Theorem \ref{thm:1}, \ref{thm:sums} in the sequel. For an exhaustive exposition of the eigenvalue method see \cite[Section 4]{s_mixed}. 

We are quite far from being able to prove the match $\E(A)\lesssim_M|A|^{7/3}$ to our support estimate \eqref{main}. In fact, even for the sum set replacing the difference set, the best known FPMS exponent is $8/5$ \cite{S_ineq}. We dedicate the last section of this paper to the sum set case, deriving a FPMS inequality with what we believe to be the best $M$-dependence in $|A+A|\gtrsim_M|A|^{8/5}$. We also remark that if we could improve the key estimate \eqref{Sig} in the proof of the energy Theorem \ref{thm:new_energy}, this would also improve the sum set exponent $8/5$. 

As some consolation, as well as possibly a principle obstacle against improving the estimates in this paper, given the technique, we take the example by Balog and Wooley \cite{BW}, also discussed in \cite{RSS}, which shows that generally the exponent $7/3$ is the best possible one for energy inequalities: there are sets $A$, such that any positive proportion subset $A'\subseteq A$ would have both $\E(A)$ and its multiplicative analogue exceeding $|A|^{7/3}$. This is unlikely to happen in the extremal FPMS case but nonetheless, together with the results in this paper bears some evidence that one can hardly expect to be able to establish $L^2$-estimates, which would be equally strong to support ones; certainly this is the case within the applicability of the techniques we possess. In fact, the methods in this paper are essentially {\em energy methods}, that is the multiplicative constant $M$ can be viewed as equal to $|A|^3/\E^\times(A)$, where $\E^\times(A)$ is multiplicative energy, and the key estimates, with some work, can be re-cast in terms of the Balog-Wooley decomposition set forth in \cite{BW}, using the state-of-the-art techniques, developed in \cite{RSS}.

We nonetheless do prove the new best known energy bound, owing largely to Lemma \ref{l:brl}, which enables us to tie together two important combinatorial characteristics of the set $A$ with small multiplicative doubling. 

$L^2$-estimates represent much interest as to many questions, arising in the context of multiplicative subgroups in $\F_p$. Our new energy bound brings an improvement to several such bounds in the literature, some being mentioned in due course.  One application we spell out explicitly in Section \ref{sec:F_p} concerns the question of maximum gap size between coset elements, Theorem \ref{gap}.

\renewcommand{\thetheorem}{\arabic{theorem}}
\medskip
Our key geometric tool is the Szemer\'edi-Trotter theorem \cite{sz-t}.
\begin{theorem} \label{SzT} Consider  a set of $n$ points in $\R^2$. Connect all pairs of distinct points by straight lines, then for $k\geq 2$, the number of lines supporting at least $k$ points is 
\begin{equation}
O\left(\frac{n^2}{k^3} + \frac{n}{k}\right).
\label{strich}\end{equation}
The total number of incidences between $n$ points  and $m$ straight lines is ${\displaystyle O(m^{2/3}n^{2/3}+m+n).}$
\end{theorem} 
In fact, in our applications the point set is a Cartesian product. In this case there is an easier proof of estimate \eqref{strich} by Solymosi and Tardos \cite{SoT}, in particular the hidden constants having very reasonable values in both the real and complex settings.

\medskip

\renewcommand{\thetheorem}{\arabic{theorem}'}
\addtocounter{theorem}{-1}
Heath-Brown and Konyagin \cite{H-K} used the Stepanov method to prove a quantitatively similar result about multiplicative subgroups in $\F^\times_p$. This was further developed in \cite{Mit}, \cite{SV}, the statement we quote can be found in \cite{Mit}.

\begin{theorem}
	\label{t:mitkin}
	Let $\Gamma$ be a multiplicative subgroup in $\F^\times_p$ and 
	$\Theta \subset \F^\times_p/\G \times \F^\times_p/\G$, 
	$|\G|^4 |\Theta| < p^3$ and $|\Theta| \le 33^{-3} |\G|^2$.  
	Then
	\begin{equation}\label{eq_t:mitkin}
	\sum_{(u,v) \in \Theta}\Bigl|\{(x,y) \in \Gamma \times \G : ux+vy=1\}\Bigr| \ll (|\Gamma||\Theta|)^{2/3} \,.
	\end{equation}
\end{theorem}


We observe, and will use, the heuristic fact that both Theorem \ref{SzT} and 
\ref{t:mitkin}
would yield the same main factor in the upper bound  $(|X||Y||Z||\Gamma|)^{2/3}$ -- see e.g. \cite[Corollary 5.1]{SV} -- for the number of solutions of the equation 
$ax+y=z$, with $a\in \Gamma$ and $x\in X,\,y\in Y,\,z\in Z$, with the restriction in $\F_p$ that the sufficiently small in terms of $p$ sets $X,Y,Z$ be $\Gamma$-invariant -- that is, say $X\Gamma=X$ -- and extra polynomial dependence in $M$ in the real case. 
For some recent work along the same lines on multiplicative subgroups in $\F_p$ see, e.g.,  \cite{Malykhin_p^2}, \cite{S_tripling}, \cite{SSV} and the references contained therein.

\renewcommand{\thetheorem}{\arabic{theorem}}

\medskip
Applications of the Szemer\'edi-Trotter theorem to sum-product type problems were started by Elekes \cite{E}, who proved the textbook inequality
 \begin{equation}\label{el}
|A\pm A|^2|AA|^2 \gg |A|^{5},
\end{equation}
which established the threshold FPMS inequality $|A\pm A|\gg_M |A|^{3/2}$.

More recently the third author \cite{S_AA} established the inequality\
\begin{equation}\label{sh}
|A-A|^6|AA|^{13}\gtrsim |A|^{23}.
\end{equation}
(In the statement  of the corresponding \cite[Theorem 1]{S_AA} one has the ratio set $A/A$, but after some inspection of the proof can be extended to embrace $AA$ as well. ) Inequality \eqref{sh}
sets the world record $|A- A|\gtrsim_M |A|^{5/3}$, which we believe is unlikely to be beaten within the current state of the art. Similarly the inequality $|\Gamma \pm \Gamma |\gg |\Gamma|^{3/2}$ for a multiplicative subgroup $\Gamma\subset \F_p$, with $|\Gamma|\leq p^{2/3}$, was established by Heath-Brown and Konyagin \cite{H-K} and improved to 
$|\Gamma- \Gamma |\gtrsim |\Gamma|^{5/3}$ for $|\Gamma|\leq \sqrt{p}$ in \cite{SV}. 


On the other hand, back to the real case, in the sense of the original question \eqref{ers}, that is when cardinalities $|AA|$ and $|A-A|$ are roughly the same, inequality \eqref{sh} is weaker that \eqref{el}.

The first inequality in the statement of the following theorem strengthens the inequality $\eqref{sh}$ (modulo the power of $\log|A|$) so that it matches the Elekes one, involving the difference set in the case of similar cardinalities. 



\begin{theorem} \label{thm:1} For a real set $A$ \begin{equation}\label{main} 
|A-A|^3|AA|^{5}\gg \frac{ |A|^{10} }{\log^{1/2} |A|}. \end{equation}
\end{theorem}
We remark that in estimates \eqref{main} one can replace $AA$ with $A/A$. The proof also applies, with $AA=A$, to the case when $A$ is replaced by a multiplicative subgroup $\Gamma\subset \F_p^\times$, with $|\Gamma|\leq \sqrt{p}$. In the latter case the inequality is due to the third author and Vyugin \cite{SV}.

\medskip
The next theorem is an $L^2$ estimate. 

\begin{theorem}
	Let $A\subset \R$ and $|AA| \le M|A|$. 
	Then
	\begin{equation}\label{f:new_energy}
	\E(A) \ll M^{8/5} |A|^{49/20} 
	 \log^{1/5}|A|\,.
	\end{equation}
	The same estimate, with $M=1$ holds for a multiplicative subgroup $\Gamma\subset \F_p^\times$,  with $|\Gamma|\leq\sqrt{p}$.\label{thm:new_energy}
\end{theorem}

In Section \ref{sec:F_p}  we develop some applications of Theorem \ref{thm:new_energy} to multiplicative subgroups in $\F_p^\times$, the main result being Theorem \ref{gap} in Section \ref{subsec:distance}.

In the final Section \ref{subsec:sums} we address sum set bounds. Theorem  \ref{thm:sums} therein establishes relatively minor improvements to known estimates; one reason why we chose to include that section is to present Lemma \ref{lemma:rect}, which can be interesting on its own accord, for until recently arguments of this type were drawn using the quantitatively costly Balog-Szemer\'edi-Gowers theorem.

\medskip
Developing efficient applications of the Szemer\'edi-Trotter theorem together with arithmetic/analytic combinatorics lemmata  has been the subject of much recent work: see, e.g. \cite{SS1},  \cite{S_ineq}, \cite{S_AA}, \cite{RSS}, \cite{SZ}, \cite{VarII} (and the references therein). 
Theorem \ref{thm:new_energy} improves the previously best known bound $\E(A) \lesssim_M |A|^{32/13}$ in  \cite[Theorem 8]{S_ineq} and  \cite[Theorem 5.4]{s_mixed}.

In conclusion of this section we remark that all the known proofs of the Szemer\'edi-Trotter theorem strongly rely on order properties of reals, and despite recent progress in incidence theory over general fields (see \cite{misha}, \cite{SdZ}) the versions of the  Szemer\'edi-Trotter theorem which apply there are weaker than Theorem \ref{SzT}. On a somewhat pessimistic note, it appears extremely unlikely that the weak Erd\H os-Szemer\'edi conjecture can be resolved over the reals without a novel insight that we currently do not possess. In particular, inequality \eqref{main}, its proof being simple as it is, appears to be the best one can hope for within today's scope of ideas.

\section{The cubic energy: basic lemmata}
In this short section we re-introduce the concept of cubic energy $\E_3(A)$, namely the third moment of the number-of-realisations function $r_{A-A}$. Generally, for $q>1$ we define
\begin{equation}
\E_q(A) := \sum_d r^q_{A-A}(d),
\label{e3def}
\end{equation}
omitting the subscript for $q=2$, where we also write $A_d=A\cap (A+d)$, as well as 
$$
\E(A,B) = \sum_d |A\cap (B+d)|^2.
$$
Returning to \eqref{e3def} we are  especially interested in the case $q=3$.

Notation-wise, if the domain of the summation index is not specified, this means the whole universe $\R$ or $\F_p$.

Geometrically $\E_3(A)$ is the number of collinear triples of points in the Cartesian product $A\times A\subset \R^2$ on unit slope  lines $y=x+d$. By looking at the projections of such a collinear point triple on the coordinate exes, the same quantity can be re-counted as
\begin{equation}\label{e3}
\E_3(A) = \sum_{d,d'} |A\cap (A+d)\cap (A+d') |^2 = \sum_{d'} \E(A,A_d) \,.
\end{equation}

I.e., triples of elements $(a,b,c)\in A\times A\times A$ get partitioned into equivalence classes by translation, a single class being identified by differences $d=b-a$, and $d'=c-a$. Two latter triples of elements of $A$ are equivalent if and only if they differ by translation. The quantity $\E_3(A)$ is the sum, over all equivalence classes, of squares of their population. (The same concerns the energy $\E(A)$, which pertains to equivalent by translation pairs, rather than triples, of elements of $A$.)

Note that since $b-c = d-d'$ is also a member of $A-A$, it follows by the Cauchy-Schwarz inequality that 
\begin{equation}\label{lower}
|\{(d,d',d'')\in (A-A)^3: \, d''=d-d'\}|\geq \frac{|A|^6}{\E_3(A)}.
\end{equation}

An application of Theorem \ref{SzT} or Theorem \ref{t:mitkin} (see, e.g., \cite[Lemma 7]{KS-})
yields a near-optimal estimate for $\E_3(A)$ with $|AA|=M|A|$ or $A=\Gamma$, quoted as part of the following lemma.

\begin{lemma}
	Suppose $|AA| \mbox{ or } |A/A|  \, = \, M|A|.$ Then for any $A'\subseteq A$, and any $B$, one has bounds
	\begin{equation}\label{b:1}\E_3(A') \ll M^2|A'|^2|A|\log|A|,\qquad \E(A,B)\ll M |A||B|^{3/2}.\end{equation}
	
	If $\Gamma$ is a multiplicative subgroup in $\F^\times_p$, with size $O(\sqrt{p})$, and $B$ a $\Gamma$-invariant set 
	$B$, then 
	\begin{equation}\label{b:2}\E_3(\Gamma) \ll |\Gamma|^3\log|\Gamma|,\qquad  \E(\Gamma,B)\ll |\Gamma||B|^{3/2}.\end{equation}
	\label{l:E_3}
	
	Bestides, for $\Delta\geq 1$
	\begin{equation}\label{b:31}
	\sum_{x:\,r_{A-A}(x)>\Delta } 1\ll \frac{ M^2|A|^3}{\Delta^3}, \qquad \sum_{x:\,r_{\Gamma-\Gamma}(x)>\Delta } 1\ll \frac{|\Gamma|^3}{\Delta^3},
	\end{equation}
	\begin{equation}\label{b:3}
	\sum_{x:\,r_{A-A}(x)>\Delta } r^2_{A-A}(x) \ll \frac{ M^2|A|^3}{\Delta}, \qquad \sum_{x:\,r_{\Gamma-\Gamma}(x)>\Delta } r^2_{\Gamma-\Gamma}(x) \ll \frac{|\Gamma|^3}{\Delta}.
	\end{equation}
\end{lemma}
\begin{remark} In fact, as far as the multiplicative subgroup case is concerned, {\em all} the inequalities of Lemma \ref{l:E_3} {\em but for}  the second inequality in \eqref{b:2} are valid for $|\Gamma|\leq p^{2/3}$. See, e.g., \cite[Lemma 4]{Ilmed}. This will be used in the proof of the forthcoming Theorem \ref{t:G_large}, where the second inequality in \eqref{b:2} will be replaced by Lemma \ref{l:M_char'}. \label{extend}\end{remark}

We use Lemma \ref{l:E_3} to immediately obtain the following consequence of estimate \eqref{lower}.
\begin{corollary}\label{cor1}
Under assumptions of Lemma \ref{l:E_3},
$$
|\{(d,d',d'')\in (A-A)^3: \, d''=d-d'\}|\gg  \frac{|A|^3}{M^2\log|A|},
$$
and the same for $\Gamma$, with $M=1$.
 \end{corollary}


\section{Proof of Theorem  \ref{thm:1}} 

Let us first prove the inequality \eqref{main} under an additional easy-to-remove assumption. Define the set of popular differences as
\begin{equation}\label{pop}
P:= \left\{ d\in A-A:\,r_{A-A}(d)  \geq \left(\Delta:= \frac{|A|^2}{2|A-A|}\right)\right\}.
\end{equation}
By the pigeonhole principle 
\begin{equation} 
\sum_{d'\in P} r_{A-A}(d') \geq \frac{1}{2}|A|^2.
\label{pig}\end{equation}

\begin{proposition} \label{cond} Suppose the bound of Corollary \ref{cor1} applies to the equation  $d''=d-d'$, with 
$d,d''\in A-A$ and $d'\in P$. Then \eqref{main} follows.
\end{proposition}

\begin{proof} This becomes merely an application of the Szemer\'edi-Trotter theorem after for any $x\in A$ one writes
\begin{equation}\label{long}\begin{aligned}
|\{(d,d',d'')\in (A-A)\times P \times(A-A): \, d=d'-d''\}| \qquad  \qquad \qquad  \qquad \qquad  \qquad \\
=  |\{(d,d',d'')\in (A-A)\times P \times(A-A): \, d''= d - xd'/x\}| \qquad  \qquad  \\
\leq \frac{1}{|A|} |\{(d,s,d'',x)\in (A-A)\times S \times(A-A)\times A: \, d''= d-s/x\}|,\;\;
\end{aligned}\end{equation}
where 
\begin{equation}\label{es}
S:=\{s\in AA-AA:\,r_{AA-AA}(s) \geq \Delta\}.
\end{equation}
Indeed, since $d'$ has at least $\Delta$ different representations $d'=b-a$ as a member of $A-A$, then $xd'$ has at least $\Delta$ the representations in the form $xd'=xb-xa$ as the difference of two products from $AA$. It follows that 
$$
|S|\leq \frac{(M |A|)^2}{\Delta}.
$$
Thus the  Szemer\'edi-Trotter theorem yields the upper bound for the number of solutions of the latter equation $d''= d-s/x$,  interpreted as incidences between $n=|A||A-A|$ points and $m=|S||A-A|$ lines, as
\begin{equation}\label{stbd}
O\left( |A|^2M^{4/3} |A-A|^{4/3} \Delta^{-2/3} \right)
\end{equation}
(it's easy to see that if the term $|\mathcal L|=|A-A||S|$ were to dominate we would have a much better estimate than \eqref{main}).

Dividing this by $|A|$ in view of \eqref{long} and comparing this with the lower bound of Corollary \ref{cor1} yields 
$$
\frac{|A|^3}{M^2\log|A|} \ll |A| M^{4/3} |A-A|^{2} |A|^{-4/3},$$ and hence \eqref{main}.

It is easy to see, after trivial modifications of the argument involving $M$ that wherever it appeared above, it could have as well come from $|A/A|=M|A|$. 
$\hfill\Box$\end{proof}

\medskip
All it takes to ensure applicability of Proposition \ref{cond} is the following, essentially trivial lemma. 
\begin{lemma} \label{l:key}
One has the following bound:
$$
\nonumber \frac{|A|^6}{4} \leq \E_3(A) |\{(d,d',d'')\in (A-A) \times P\times(A-A) : \, d''= d-d'\}|.$$
\end{lemma}
\begin{proof} The proof is just an application of the pigeonhole principle, often presented as a  graph density argument. Considering $P$ as a bipartite graph on $A\times A$, it has density $\geq 1/2$. Therefore, the number of triples $(a,b,c)$ with $a-b\in P$ is $\geq |A|^3/2$. The claim of the lemma now follows by the Cauchy-Schwarz inequality, just like \eqref{lower} above.
$\hfill\Box$\end{proof}



\section{Proof of Theorem \ref{thm:new_energy} }
\label{subsec:T_3}

The proof of Theorem \ref{thm:new_energy} mainly rests on lemmata \ref{l:brl} and \ref{l:spectral}.

In general,  for $k\ge 2$ an integer, and $A$ a subset of an abelian group, 
let $\T_k (A)$ (such additive characteristics of sets appear throughout additive combinatorics literature) be the quantity
\begin{equation}\label{tkdef}
	\T_k (A) := |\{ (a_1,\dots,a_k,a'_1,\dots,a'_k)\in A^{2k} ~:~ a_1 + \dots +a_k = a'_1 + \dots +a'_k \}| \,.
\end{equation}
Clearly $\T_2 (A) = \E (A)$ is additive energy; we here focus on $\T_3(A)$.
We have also denoted by 
$\mathcal T(A)$ the number of {\it collinear point triples} in the plane set $A\times A$. Note that Theorems \ref{SzT} and \ref{t:mitkin}, applied respectively in the real and $|\Gamma|\leq\sqrt{p}$ multiplicative subgroup $\Gamma\subset \F^\times_p$ settings, ensure that 

\begin{equation} \label{trip} \mathcal T(A)\ll |A|^4\log|A|. \end{equation}.

\medskip

The next lemma evinces a connection between the quantities $\T_3 (A)$ and $\mathcal T(A)$. 

\begin{lemma}\label{l:brl}
	One has the inequality
$$
	\T_3 (A) \le  \frac{|AA/A|}{|A|^2}\min\left(|AA| \sqrt{\mathcal T(A/AA)\cdot \mathcal T(AA)},\;  |A/A|\sqrt{\mathcal T(AA/A)\cdot \mathcal T(A/A)}\right) \,.
$$
\label{l:T_3}
\end{lemma}

Using the standard Pl\"unnecke inequality, that is if $|AA|$ or $|A/A|$ is $\leq M|A|$, then $|AA/A|\leq M^3|A|$ (see e.g. \cite{TV})  and the Szemer\'edi-Trotter Theorem \ref{SzT} or Theorem \ref{t:mitkin}, we arrive in the following corollary. 
\begin{corollary}\label{brlc} If $|AA|\leq M|A|$ or $|A/A| \leq M|A|$, then
$$
\T_3 (A)\ll M^{12}|A|^4 \log|A|,
$$
the same holds with $M=1$ if $A$ is replaced by $\Gamma$, a multiplicative subgroup in $\F^\times_p$ with $O(\sqrt{p})$ elements.
\end{corollary}
We remark that in the context of multiplicative subgroups $\G$, when $|\G|=O(\sqrt{p})$ the bound $\mathcal T(\G)\ll |\Gamma|^4\log|\Gamma|$ can be found in \cite[Proposition 1]{S_tripling}.

\begin{remark}\label{rem:1}

Corollary  \ref{brlc} (in the case of small multiplicative subgroups) improves some results of the fourth  author, see \cite{Iur}.
Upper bounds for 
$\T_3(\G)$ have interesting applications to number-theoretic congruences studied, e.g. by  Cilleruelo and Garaev  \cite{CG1}, \cite{CG2}. 

2. Consider the case when $\G \subset (\mathbb{Z}/p^{2}\mathbb{Z})^{*},$ and $|\G|$ divides $p-1$. One can consider the subgroup $\G' \subseteq \F_p$ where $\G' = \G \pmod p$. It is easy to check that $|\G'|=|\G|$ and hence $\T_{k}(\G)\leq \T_{k}(\G')$. Using this and the method of the proof of Lemma \ref{l:brl} and \cite[Proposition 1]{S_tripling},  one can deduce the following statement.

\end{remark}

\begin{lemma}\label{P2}
Let $\G$ be subgroup of $(\mathbb{Z}/p^{2}\mathbb{Z})^\times$ and let $|\G|=t$. If $t \leq p^{1/2}$ then $\T_{3}(\G) \lesssim t^{4}$. If $t=p^{1/2 + \delta}$, then $\T_{3}(\G) \lesssim t^{4+6\delta}$.	
	
\end{lemma}

The reader can also find the previous bounds for $\T_{3}(\G)$ in the paper \cite{Malykhin_p^2}, which were obtained purely by Stepanov's method. 
To avoid repeating the above-mentioned proofs, we content ourselves with merely a sketch of the proof of Lemma \ref{P2} further in this paper.

\medskip
In the forthcoming proofs we avoid keeping track of exact powers of $M$ and $\log|A|$ to make the formulae shorter; the reader is invited to check that they are indeed as presented in the statement of the theorem. The  next lemma is crucial; its predecessor can be found in \cite[Section 4]{s_mixed}.

\begin{lemma} Let $1\leq\Delta\leq |A|$ and
$$\E(A) = \E'+\E'':= \sum_{x:r_{A-A}(x)\leq \Delta}r_{A-A}^2(x) \;+  \sum_{x:r_{A-A}(x)>\Delta}r_{A-A}^2(x).$$
	One has the inequality 
	\begin{equation}\label{f:E_3_2}
	\E'^6  
		\leq  
		|A|^{6} \E_3 (A) \Delta^{2} \Sigma,  
	\end{equation}
	where 
	\begin{equation}\label{sigma}
	\Sigma:=\sum_{d,d'} r_{A-A}(d) r_{A-A}(d') r^2_{A-A}(d-d').
	\end{equation}
	\label{l:spectral}
\end{lemma}
We are now ready to prove  Theorem \ref{thm:new_energy}.

\medskip
\begin{proof} We use Lemma \ref{l:spectral} with the choice of $\Delta \sim_{M} |A|^{11/20}$ (that is we omit the exact powers of $M$ and $\log|A|$). Then we are done with the proof of Theorem \ref{thm:new_energy} if $\E''\gg \E(A)$, owing to estimate \eqref{b:3} of Lemma \ref{l:E_3}.

We proceed assuming that $\E'\gg\E(A)$. Denote $P:=A-A.$ 
		
	The quantity $\Sigma$ in \eqref{sigma}
	counts solutions of the equation
	\begin{equation}\label{brr}
	a-a'=b-b'=(a_1-a_2)-(a_3-a_4),\end{equation}
	with all the variables in $A$.

	Let us introduce a cutoff parameter $\tau \sim_M |A|^{3/5}$ whose value is to be justified shortly. We now partition $P=P'\cup P''$ the set $A-A$ into the set of ``poor'' and ``rich'' elements, namely
	$$
	P'=\{d:\,r_{A-A}(d)\leq\tau\},\qquad P''=P\setminus P'.
	$$
	Correspondingly,
	$$
	\Sigma=\Sigma'+\Sigma'',
	$$
	where $\Sigma'$ is the restriction of the count \eqref{brr} to the case when $(d=a-a'=b-b' )\in P'$.
		
	Clearly, 
	\begin{equation}\label{sigma1} \begin{aligned} \Sigma' 
	&\leq \tau |\{a_1-a_2=(a_3-a_4)-(a_5-a_6):\,a_1,\ldots, a_6\in A\}|\;= \;
	  \tau \T_3(A) \\ & \leq_M \; \tau \mathcal T(A) \\ & \lesssim \; \tau|A|^4,\end{aligned}
	\end{equation}
	by Lemma \ref{l:brl} and \eqref{trip}.

	On the other hand, for $P''=\{d: \,r_{A-A}(d)>\tau\}$, Lemma \ref{l:E_3} provides a cardinality bound, decreasing as $\tau^{-3}$. We need a slightly more elaboration, a dyadic partitioning to be soon summed as a vanishing geometric progression to show that as to $P''$, one can roughly assume that $r_{A-A}(d)\ll \tau$, for all $d\in P''$. 
	
	Namely,  $j\geq 1$ set
	 $$P''_j:=\{d:\, 2^{j-1}\tau \leq r_{A-A}(d)<2^{j}\tau\}.$$
	Then, denoting for simplicity $2^j\tau =\tau_j$, one has, by \eqref{b:31},
	\begin{equation}
	|P_j''|\ll \frac{M^2|A|^3}{\tau_j^3}.
	\label{pej}\end{equation}
	Setting $\Sigma''_j$ to be the corresponding to $P''_j$ component of $\Sigma''$,  that is when the sum in \eqref{sigma}  is restricted to $d-d'\in P''_j$, we can bound
	$$\begin{aligned} \Sigma_j'' &\leq \tau_j^2 |\{(d,a_1,a_2,a_3,a_4)\in P''_j\times A\times\ldots\times A):\, d=(a_1-a_2) -(a_3-a_4)\}| \\
	&\leq \tau_j^2 \sqrt{\E(A, P''_j)}\sqrt{\T_3(A)}, 	\end{aligned}
	$$
	by Cauchy-Schwarz. 
	
	We substitute \eqref{pej} to claim, by \eqref{b:1}, Lemma \ref{l:E_3}:
	$$
	\E(P''_j,A)\ll  M |A||P_j|^{3/2} \ll_M  |A|^{11/2} \tau_j^{-9/2}.
	$$
	Furthermore, $\T_3(A)\lesssim_M |A|^4$ as above,
	by Lemma \ref{l:brl} and \eqref{trip}.
	
	Thus
	$$
	\Sigma'' \lesssim_M \tau^{-1/4}|A|^{19/4} \sum_{j\geq 1} 2^{-j/4} \ll \tau^{-1/4}|A|^{19/4}.
	$$
	We now match the latter estimate and \eqref{sigma1} for $\Sigma'$: this prompts the choice $\tau \sim_M|A|^{3/5}$ (that is up to powers of $M$ and $\log|A|$) and proves that
	\begin{equation}\label{Sig}
	\Sigma \lesssim_M |A|^{23/5}.
	\end{equation}

	We now go back to the main estimate \eqref{f:E_3_2} of Lemma \ref{l:spectral}. We have, by Lemma \ref{l:E_3}, the fact that $\E_3(A)\lesssim_M |A|^3$ 
	and the assumption $\E(A)\gg \E'$ in the statement of Lemma \ref{l:spectral} ends the proof  of  Theorem \ref{thm:new_energy} after substituting $\Delta=|A|^{11/20}$ in the lemma's estimate. Indeed, we get
	$$
	\E^6(A) \lesssim_M |A|^6 \cdot |A|^3 \cdot|A|^{11/10} \cdot |A|^{23/5} = |A|^{147/10}.
	$$
	\end{proof}
	$\hfill\Box$

\section{Proofs of main lemmata}

\subsection{Proof of Lemma \ref{l:brl}}

\begin{proof}
	Clearly, for any triple $(h_1,h_2,h_3) \in A\times A \times A$ and $a\in A$
	we have
	$$
	h_1-h_2-h_3=(h_1-h_2)\left(1 - \frac{h_3}{h_1}\right) - \frac{h_2h_3}{h_1} = (ah_1-ah_2)\left(a^{-1} - a^{-1}\frac{h_3}{h_1}\right)- \frac{h_2h_3}{h_1}.
	$$
	Set $\alpha=a^{-1}\frac{h_3}{h_1}\in A/AA$, $\beta= ah_2\in AA$. Then we have the following estimate, where the first line follows from the latter identity, and in the second line the Cauchy--Schwarz inequality and interchange of the order of summation have been applied:
\begin{align*}
	\T_3 (A)  = \sum_x r^2_{A-2A} (x)  & \le |A|^{-2} \sum_x \left( \sum_{\a\in A/AA,\, \beta \in AA} r_{(A^{-1}-\a) (AA-\beta)} (x+\a \beta) \right)^2 \\ 
&
		\le
			\frac{|AA| |A/AA|}{|A|^2}  \sum_{\a\in A/AA,\, \beta \in AA}\, \sum_x r^2_{(A^{-1}-\a) (AA-\beta)} (x).	
				\end{align*}
The three-index sum in the right-hand side is the number of solutions of the equation
$$(b-\a)(c-\beta) = (b'-\a)(c'-\beta):\qquad \a\in A/AA,\, \beta \in AA,\, b,b' \in A^{-1}, \,c,c' \in AA,$$
or, after rearranging, of the equation (with the same variables)
\begin{equation}\label{f:r_to_T}
\frac{b-\a}{b'-\a} = \frac{c'-\beta}{c-\beta}.
\end{equation} 
The left-hand side of the latter equation, since $A^{-1}\subseteq A/AA$, has all its variables  $b,b',\a\in A/AA,$ the right-hand side $c,c',\beta\in AA$. Applying once again the Cauchy-Schwarz inequality, the number of solutions of the latter equation is bounded by $\sqrt{\mathcal{T}(A/AA)\cdot \mathcal{T}(AA)}$.

This completes the proof, once we note that $|AA/A|=|A/AA|$ and that one could implement the above procedure for any $a\in A^{-1}$, rather than $a\in A$.
$\hfill\Box$
\end{proof}

\subsection{Sketch of Proof of Lemma \ref{P2}}

According to Lemma \ref{l:brl} we need to find upper bound for the number of solution to the equation 
$$
(a-b)(a'-c')=(a-c)(a'-b'):\; a,b,c,a',b',c' \in \G' \equiv \G \pmod p.
$$
 Here we follow the scheme of the proof of Proposition 1 of the paper \cite{S_tripling}.
 It is easy to see that for any tuple ($a,b,c,a',b',c'$) satisfying the above equation, the points $(a,a'), (b,b'), (c,c')$ lie on the same line and one can  assume that these points are pairwise distinct. One can restrict the set of lines to  only those in the form 
 $$ux+vy=1.$$
Define $l_{u,v}=|\{(x,y)\in \G' \times \G' : ux+vy=1\ \}|$. So, we need to get an upper estimate for the sum 
 $$
 \sum_{u,v}l_{u,v}^{3} \,.
 $$

Such an estimate follows from Theorem \ref{t:mitkin} after easy  calculations.
For the case of $|\G|=O(\sqrt{p})$ this method gives a near-optimal estimate, when $\delta\neq 0$,  the claim of Lemma \ref{P2}
follows by  application of the H\"older inequality.


\subsection{Proof of Lemma \ref{l:spectral}} \label{section:spec}
 The proof represents an instance of the eigenvalue method developed by the third  author. 
 
 \begin{proof} 
 
 Consider a $|A|\times|A|$ matrix $\mathfrak M$, with elements $\mathfrak M_{ab}=\sqrt{r_{A-A}(a-b)}$. In addition, let $\mathfrak R$ be the matrix with  entries 
$\mathfrak R_{ab}=r_{A-A}(a-b).$

We observe that both matrices have positive entries, are symmetric, and the matrix $\mathfrak R$ is semipositive-definite.
Indeed, for any vector $\ve\in \R^{|A|}$, identifying $A$ with its characteristic function, the same for its shifts, say $A+(a-b)$ below, we have, after a rearrangement
$$\begin{aligned}
\ve\cdot \mathfrak R\ve & = \sum_{a,b,c } A(c) [ A+(a-b)](c) v_av_b \\ & = \sum_{a,b,c }  [ A-b](c-a) [A-a](c-a)v_a v_b\\
&= \sum_{x} \left(\sum_a [A-a](x)v_a\right)^2.
\end{aligned}
$$
 
Let us calculate the trace $\tr(\mathfrak M^2\mathfrak R)$ in two different bases. In the standard basis
\begin{equation}\label{longer}\begin{aligned}
\tr(\mathfrak M^2\mathfrak R) & = \sum_{x,y,z\in A} \sqrt{r_{A-A}(x-y)} \sqrt{r_{A-A}(x-z)} r_{A-A}(x-z) \\
&=\sum_{d,d'} \sqrt{r_{A-A}(d)} \sqrt{r_{A-A}(d')} r_{A-A}(d-d') |A\cap(A+d)\cap(A+d')| \\
& \leq
\left( \sum_{d,d'} |A\cap(A+d)\cap(A+d')| ^2\right)^{1/2} \left( \sum_{d,d'} r_{A-A}(d) r_{A-A}(d') r^2_{A-A}(d-d')  \right)^{1/2} \\ &= 
\sqrt{\E_3(A) \Sigma}.
\end{aligned}\end{equation}
Modulo a power of $|A|$ this gives the square root of the  right-hand side in the lemma's estimate.

Let us now  estimate $\tr(\mathfrak M^2\mathfrak R)$ from below. Since both matrices have positive entries, the trace will not increase if we zero the elements $\mathfrak M_{ab}$, with $r_{A-A}(a-b)>\Delta$. 

Let us now define the matrix $\mathfrak M'$ (see the forthcoming Remark \ref{rem:2} as to why) as follows:
\begin{equation}\label{Mprime}
 \mathfrak M'_{ab} = \left\{ \begin{array}{ll}\Delta^{-1/2}\mathfrak R_{ab},\mbox{ if } r_{A-A}(a-b)\leq \Delta,
 \\ 0\mbox{ otherwise. } \end{array}\right.
\end{equation}
The matrix $\mathfrak M'$ is clearly symmetric, and we have an entry-wise bound $\mathfrak M'_{ab}\leq  \mathfrak M_{ab}$ for all $a,b\in A$.

We now get a lower bound on  $\tr(\mathfrak {M'}^2\mathfrak R)$.

Consider the orthonormal basis $\{\ve_1,\ldots,\ve_{|A|}\}$ of real eigenvectors of $\mathfrak M'$, the corresponding real eigenvalues being 
$\mu_1,\ldots \mu_{|A|}$, ordered by non-increasing moduli. The eigenvector $\ve_1$, corresponding to the principal eigenvalue $\mu_1$ is non-negative by the Perron-Frobenius theorem.

Hence, since $\mathfrak R$ is semipositive-definite
$$
\tr(\mathfrak {M'}^2\mathfrak R) = \sum_{a\in A} \mu^2_a (\mathfrak R\ve_a\cdot\ve_a) \geq  \mu^2_1(\mathfrak R\ve_1\cdot\ve_1).
$$
Since $\mu_1$ is the spectral radius of $\mathfrak {M'}$, we have, with $\ve=\frac{1}{\sqrt{|A|}}(1,\ldots,1)$,
$$\mu_1=\ve_1\cdot \mathfrak {M'}\ve_1\geq \ve\cdot \mathfrak {M'}\ve = \Delta^{-1/2}\E'.$$
Furthermore, by non-negativity of $\mathfrak {R}$ and $\ve_1$ and the previous estimate,
$$
\ve_1\cdot \mathfrak {R}\ve_1 \geq \Delta^{1/2} \ve_1\cdot \mathfrak {M'}\ve_1  \geq \E'.
$$

This completes the proof  of Lemma \ref{l:spectral}.
$\hfill\Box$
\end{proof}

\begin{remark}\label{rem:2}
In the case of a multiplicative subgroup $\Gamma$, the proof could be made slightly more straightforward, for the reason from passing from $\mathfrak M$ to $\mathfrak M'$ in \eqref{Mprime} was that otherwise we would need additional argument as to how $\mathfrak R\ve_1\cdot\ve_1$, where $\ve_1$ is the principal eigenvector of $\mathfrak M$, compares to $\mathfrak R\ve\cdot\ve$, with  $\ve=\frac{1}{\sqrt{|A|}}(1,\ldots,1)$. In the multiplicative subgroup case $\ve=\ve_1$, for both matrices $\mathfrak M$ and $\mathfrak R$ are regular (in fact, circulant). This implies the lower bound bound
$$\tr(\mathfrak M^2\mathfrak R)\geq |A|^{-3} \E_{3/2}^2(A)\E(A),$$ where the notation $\E_{3/2}(A)$ was defined by \eqref{e3def}. The latter estimate suffices to yield Theorem \ref{thm:new_energy}, for by the H\"older inequality
$$
\E(A)\leq \E_3(A)\E_{3/2}^{2/3}(A).
$$
\end{remark}


\bigskip

Previously best known energy bounds that Theorem \ref{thm:new_energy} improves on, have been used in many papers, quoted in the introduction and beyond.
The new bound  (\ref{f:new_energy}) automatically results in improvement of estimates, which relied on its predecessors. This concerns, in particular, the results in \cite{SZ}, dealing with multiplicative energy of sumsets.

\section{Additive energy of multiplicative subgroups}
\label{sec:F_p}


Studying translation properties of multiplicative subgroups in $\F_p^\times$ is a classical subject of number theory see, e.g., \cite{KS1}, \cite{SS}.
In \cite{H-K,K_Tula} it was proved, in particular, that $\E(\G) \ll |\G|^{5/2}$ for any subgroup $\G\subset \F_p$,  with $|\G| \leq p^{2/3}$.  
It is well known that the latter bound is optimal for subgroups of size $\Omega(p^{2/3})$. 
In \cite{S_ineq} better energy bounds for multiplicative subgroups of smaller size were obtained. 
Theorem \ref{thm:new_energy} sets the new record for $|\Gamma|=O(p^{1/2})$. It also allows for some improvement of the ``intermediate range'' bounds for $\Omega(p^{1/2})= |\Gamma|=O(p^{2/3})$, presented in this section.

Let us restate the results that we are going to use. See the beginning of Section \ref{subsec:T_3} for some of the notation used.
	
\begin{corollary}
	Let  $\Gamma \subseteq \F^\times_p$ be a multiplicative subgroup, $|\G| \le \sqrt{p}$. 
	Then
\begin{equation}\label{f:subgr_1}
	\T_3 (\G) \leq \mathcal T (\Gamma) \ll |\G|^4 \log |\G| \,, 
\end{equation}
	and 
\begin{equation}\label{f:subgr_2}
	\E(\G) \ll |\G|^{49/20} \log^{1/5} |\G|\,.
\end{equation}
\end{corollary}

To get new energy bounds for intermediate size subgroups, we apply the standard technique from the literature cited below and the following theorem \cite[Theorem 1.2]{MSS} which replaces the bound \eqref{f:subgr_1}.

\begin{theorem}
	Let  $\G\subset \F^\times_p$ be an arbitrary multiplicative subgroup. 
	Then
$$
\mathcal{T} (\G) - \frac{|\G|^6}{p}  \ll 
\left\{
\begin{array}{ll}
p^{1/2} |\G|^{7/2},& \text{if $ |\G| \ge p^{2/3}$},\\
|\G|^5 p^{-1/2} , & \text{if $p^{2/3} > |\G| \ge p^{1/2}\log p$},\\
|\G|^4 \log |\G|, & \text{if $|\G|< p^{1/2}\log p$}. 
\end{array}
\right.
$$
\label{t:T(G)_Shp}
\end{theorem}

\begin{remark}
	Using Theorem \ref{t:T(G)_Shp} one can sightly improve the upper bound for $\T_3 (\G)$, $\G\subseteq \Z/p^2 \Z$ in Lemma \ref{P2}; we leave this to a keen reader.
\end{remark}

The changes to estimate \eqref{f:subgr_2} we are about to address are also due to the fact that the estimates of Theorem \ref{t:T(G)_Shp} have replaced  \eqref{f:subgr_1}, and besides that the full analogue of the second inequality in \eqref{b:2} from Lemma \ref{l:E_3} (used once in the proof of \eqref{f:subgr_2}) for subgroups  $|\G| \le p^{2/3}$ is as follows.

\begin{lemma}
	Let $\G \subset \F^\times_p$ be a multiplicative subgroup, $|\G| \le p^{2/3}$.  
	Then
	and for any $\G$--invariant set $Q \subset \F^\times_p$  (i.e. $Q\G = Q$) one has 
	\begin{equation}\label{f:M_char'}
	\E(Q,\G) \ll \frac{|\G|^2 |Q|^2}{p} + |\G| |Q|^{3/2} \,.
	\end{equation}
	\label{l:M_char'}
\end{lemma}

Namely for $|\Gamma|=\Omega(p^{1/2})$, one has to add to the second estimate in \ref{b:2} the ``statistical average'', see, e.g., \cite{SS1} or \cite{S_ineq}. Note that the format of second term in  \eqref{f:M_char'}
complies with using the point-plane theorem of \cite{misha}, see also \cite{RRS}. 

\bigskip

We now present the new energy bound, improving the bound $\E(\G) \ll |\G|^{5/2}$ (sharp for $|\G| \sim p^{2/3}$)
for subgroups with 
$|\G| \lesssim p^{5/8}$.

\begin{theorem}
		Let  $\Gamma \subseteq \F^\times_p$ be a multiplicative subgroup, with $p^{1/2} \le |\G| \le p^{2/3}$. 
	Then
	\begin{equation}\label{f:G_large}
	\E(\G) \ll \log^{1/4} |\G| \cdot \max \left\{ \left( \frac{|\G|^{104}}{p^3} \right)^{1/40},\; \left( \frac{|\G|^{68}}{p^5} \right)^{1/24} \right\}\,.
	\end{equation}
\label{t:G_large}
\end{theorem}
{\it Outline of the proof.~}
Set $L=\log |\G|$. 
We repeat the arguments of the proof of Theorem \ref{thm:new_energy}, estimating $\mathcal{T}(\G)$ as $\mathcal{T}(\G) \ll L |\G|^6/p$, according to Theorem \ref{t:T(G)_Shp}. 
To bound energies $\E(P''_j,\G)$ which appear in the proof, see estimate \eqref{pej} and argument following it,
we use the estimate  \eqref{f:M_char'} of Lemma \ref{l:M_char'}. If the second term in the application of (\ref{f:M_char'}) dominates, we literally repeat the proof of Theorem \ref{thm:new_energy}, obtaining $$\E(\G) \ll \log^{1/5} |\G| \cdot \left( \frac{|\G|^{104}}{p^3} \right)^{1/40}.$$ Observe that owing to Remark \ref{extend} the estimates of Lemma \ref{l:E_3} on $\E_3(\Gamma)$, as well \eqref{pej} reman valid. 

It is easy to see that apart from the above estimate on $\mathcal{T}(\G)$, the alternative case of the estimate of Lemma \ref{l:M_char'} is the only modification to the proof of Theorem \ref{thm:new_energy} required. A straightforward calculation  leads to choosing in the later case the value of the parameter $\tau$ in the proof of Theorem \ref{thm:new_energy} as  $\tau =  |\G|^{5/3} p^{-2/3}$, which then yields the  inequality
$$
	\E_{3/2}^4(\G) \E^2 (\G) \ll L |\G|^9 \mathcal{T}(\G) \cdot \tau \ll  L^2 |\G|^9 \frac{|\G|^6}{p} \cdot \tau \,.
$$
Substituting $\tau =  |\G|^{5/3} p^{-2/3}$  and using the H\"older inequality to get rid of $\E_{3/2}^4(\G),$ see Remark \ref{rem:2}, we have
$$
	\E^6 (\G) \ll L^2 |\G|^{15} p^{-1} (|\G|^{5/3} / p^{2/3}) \cdot (|\G|^3/\E(\G))^2 \,.
$$
and therefore
$$
	\E (\G) \ll L^{1/4} \cdot \left( \frac{|\G|^{68}}{p^5} \right)^{1/24} \,.
$$
\bigskip 

Bound (\ref{f:G_large})
is better than the previously best known one in \cite[Theorem 8]{S_ineq} for subgroups of size 
$p^{1/2} \le |\G| \lesssim p^{4/7}$. 

\subsection{On the greatest distance between the adjacent elements of cosets of a subgroup.}	
\label{subsec:distance}

Following Bourgain, Konyagin and Shparlinski  \cite{bur3}, we introduce, for a multiplicative subgroup $\G \subseteq \F^\times_p$ of order $t$,  the maximum gap $H_{p}(t)$ between elements  of cosets of $\Gamma$, as follows:

$$H_{p}(t)= \max\{H: \exists a \in \mathbb \F_{p}^{*}, 
\exists u \in \mathbb \F_{p}, 1 \leq j \leq H:
 u+j \in \mathbb \F_{p} \setminus a\Gamma  \}.$$

\medskip
In \cite[Theorem 3]{bur3}  the following bound was established.

\begin{theorem}
	For $t \geq p^{1/2}$, one has
$$
H_{p}(t) \leq p^{463/504 + o(1)}, \;p \rightarrow \infty .
$$ 

\end{theorem}

The case $t \geq p^{1/2}$ is important, because for any $g>1$  and for almost all $p$   the subgroup generated by powers of $g$ 
has cardinality at least $p^{1/2}$, see \cite{pap}. The distribution of the elements of this  subgroup is closely related to the distribution of digits of $1/p$ in base $g$.

We use the symbol $o(1)$ in this section to subsume terms which are smaller than any power of $p$, most of these terms come from the forthcoming quote of \cite[Theorem 1]{bur3} as Theorem \ref{subgr_int} here, to be used as a black box.

The above exponent $\frac{463}{504}$ was  improved to $\frac{5977}{6552}$ in \cite{Iur}.
New estimates for additive energy of multiplicative subgroups allow for further improvement, as follows.

\begin{theorem}\label{gap}
	For $t \geq p^{1/2}$, one has
$$
H_{p}(t) \leq p^{\frac{437}{480} + o(1)}, \; p \rightarrow \infty .
$$ 

\end{theorem}

Before proving Theorem \ref{gap}, let  us  introduce several auxiliary quantities. Let $g$ be the primitive root  of $\mathbb F^\times_{p}.$ $\G\subseteq \F^\times_{p}$, as we said,  is a multiplicative subgroup of order $t$; set $n=(p-1)/t$. 
Also let $\G_{j}:=g^{j}\G$ and 
$$ 
S_{j}(t):= S(g^{j}, \G)=\sum_{x \in \G}e_{p}(g^{j}x); \quad \quad 
N_{j,t}(h):=|\{1 \leq |u|  \leq h : u \in \G_j \}|,
$$
where $e_p$ is the canonical additive character.

The quantities $H_{p}(t), N_{j,t}(h)$ and $S_{j}(t)$ are related via following statement \cite[Lemma 7.1]{KS1}.

\begin{theorem}
If	for some $h\geq 1$ the inequality :
$$
\sum_{1 \leq j \leq n}N_{j,t}(h)|S_{j+k}(t)| \leq 0.5 t 
$$ 
holds for all $k=1, \ldots , n$, then for any $\varepsilon>0$ 
$$
H_{p}(t) \ll p^{1+\varepsilon}h^{-1}.
$$
\end{theorem}

Besides, it is easy to see that the quantity 
$$N(\G,h):= \sum_{1 \leq j \leq n}N_{j,t}^{2}(h)$$ is the number of solutions  to the congruence 
$$
 ux \equiv y \pmod p\,, \quad \quad 0<|x|,|y| \leq h, u \in \Gamma  \,.
$$
In \cite[Theorem 1]{bur3}, an upper bound for $N(\G,h)$ was proved in the following form.

\begin{theorem}\label{subgr_int}
Let $\nu \geq 1$ be a fixed integer and let $t \gg p^{1/2}, \;p \rightarrow \infty.$ Then 
$$
N(\G,h) \leq h t^{\frac{2\nu +1}{2\nu(\nu+1)}}p^{\frac{-1}{2(\nu + 1)}+o(1)} + h^{2}t^{1/\nu}p^{-1/\nu + o(1)}.
$$
\end{theorem}

By orthogonality, it follows from definition of the quantity $S_{j}(t)$ that
\begin{equation}\label{ort}\sum_{1 \leq j \leq n} |S_{j}(t)|^{4} < \frac{p}{t} \E(\G).\end{equation}

\medskip
We are now ready to prove Theorem \ref{gap}. 

\medskip
\begin{proof} The structure of the proof repeats its predecessors in \cite{bur3} or \cite{Iur}.
	Just as above $t:=|\G|$. By the H\"older  inequality we obtain
	
	$$\sum_{1 \leq j \leq n}N_{j,t}(h)|S_{j+k}(t)| \leq \biggl(\sum_{1 \leq j \leq n} N_{j,t}(h)\biggr)^{1/2} \biggl(\sum_{1 \leq j \leq n} N_{j,t}(h)^{2}\biggr)^{1/4} \biggl(\sum_{1 \leq j \leq n} |S_{j}(t)|^{4}\biggr)^{1/4} .$$
	
	For the three terms in the righthand side we have estimates 
	$$\sum_{1 \leq j \leq n} N_{j,t}(h)= 2h,$$
	$$\sum_{1 \leq j \leq n} N_{j,t}(h)^{2}=N(\G,h),$$ 
	and \eqref{ort}.

	Define $h=p^{43/480 -\varepsilon}$ for some small fixed $\varepsilon>0$.
	
	Consider the case  $t \in [p^{1/2}, p^{4/7}]$. Then by \eqref{f:G_large} one has 
	$$\E (\G) \lesssim \left(\frac{|\G|^{104}}{p^{3}}\right)^{1/40}.$$
	We take $\nu=6$ in the estimate of Theorem \ref{subgr_int} for $N(\G,h)$, in which case the second term in the estimate dominates.
	One can easily check that for such choice of parameters, the quantity 
	\begin{equation}
		\sum_{1 \leq j \leq n}N_{j,t}(h)|S_{j+k}(t)|
	\end{equation}
	is less than $0.5t$.
	
	Now consider the case when $t>p^{4/7}$. Then we merely write 
	$$
		\sum_{1 \leq j \leq n}N_{j,t}(h)|S_{j+k}(t)| \leq \max_{j}|S_{j}(t)| \sum_{1 \leq j \leq n} N_{j,t}(h) \leq p^{1/6+o(1)}t^{1/2}h. 
	$$
	The last inequality took advantage of the supremum estimate for $|S_{j}(t)|$ that we take from \cite[Theorem 1]{Ilmed}. It is easy to verify that  the inequality 
	\begin{equation}
		\sum_{1 \leq j \leq n}N_{j,t}(h)|S_{j+k}(t)| \leq 0.5t 
	\end{equation}
	also holds. With that, and taking sufficiently small $\varepsilon>0$, the proof of Theorem \ref{gap} is completed.
	$\hfill\Box$
\end{proof}

\section{Sum set inequalities} \label{subsec:sums} We start out with a remark that one can repeat the proof of Theorem \ref{thm:new_energy} with the matrix $\mathfrak M$ being replaced by the matrix of zeroes and ones, with ${\mathfrak M}_{ab}=1$ if the sum $a+b$ is popular, that is
$r_{A+A}(a+b)\geq \frac{|A|^2}{2|A+A|}.$

This will not alter the structure of the proof, owing the identity
$$
(b+a)-(c+a)=b-c, \;\mbox{ replacing } \; (b-a)-(c-a)=b-c,
$$
as well as the fact that the quantity $\E^+(A):=\sum_s r^3_{A+A}(s)$ satisfies the same upper bound as $\E_3(A)$ in \eqref{b:1}, Lemma \ref{l:E_3} -- it would replace $\E_3(A)$ in the analogue of estimate \eqref{longer}.

After that the acme of getting the upper bound on $\tr(\mathfrak M^2\mathfrak R)$ would remain to be estimate \eqref{Sig}.  However, estimating $\tr(\mathfrak M^2\mathfrak R)$ from below does present some challenge in the real case, owing to Remark \ref{rem:2}. However, in the multiplicative subgroup case the matrices $\mathfrak M,\,\mathfrak R$ still have the same principal eigenvector if al ones, and one arrives in the following estimate: for $|\Gamma|\leq\sqrt{p}$,
$$
|\Gamma+\Gamma| \gg |\Gamma|^{8/5} \log^{-2/5}|\Gamma|.
$$
The same estimate, with a slightly worse power of $\log|\Gamma|$ was established in \cite{S_ineq}. We do not know how to improve the exponent $8/5$ for the sum set towards $7/3$ as we have for the difference set but would like to point out that both the sum set and the energy FPMS exponents have been now made dependent on the same estimate \eqref{Sig}.

To circumnavigate the difficulty arising from the fact that principal eigenvectors of the matrices with elements $r_{A+A}(a+b)$ and $r_{A-A}(a-b)$ may be different for $M>1$, we present a different proof, resulting in a new FPMS-type sum-product inequality, with a ``reasonable'' $M$-dependence. Heuristically speaking, the better the FPMS-exponent, the weaker the inequality, treated as the usual sum-product inequality, namely when $|AA|\sim|A+A|$, cf. \eqref{sh}. In this sense, the following inequality \eqref{sum_est} is slightly weaker than the classical Elekes one \eqref{el} and represents, in some sense a sum set analogue of \eqref{sh}. 

We remark that on can find ``middle-ground'', that is non-trivial from the FPMS point of view (that is giving a FPMS exponent greater than $3/2$) and still stronger than \eqref{el}, due, e.g., to Li and Roche-Newton \cite{Li_R-N} or the one used by Konyagin and Shkredov \cite[Theorem 12]{KS} to derive their new sum-product bound. But as the FPMS aspect improves, the other aspect eventually becomes weaker than \eqref{el}, the only exception (up to a log factor) being \eqref{main}.

\begin{theorem} \label{thm:sums} For a real set $A$ one has the estimates \begin{equation}\label{sum_est} 
 \qquad 
|A+A|^{10}|AA|^{17}\gtrsim |A|^{33}.
\end{equation}
\end{theorem}

The forthcoming proof  of Theorem \ref{thm:sums} uses the construction from \cite[Proof of Theorem 10]{S_AA}, which dealt with differences and led to \eqref{sh}. Its adaptation to sums rests on the following Lemma \ref{lemma:rect} --  a stronger version of the argument  from  \cite[Proof of Proposition 16]{RSS} and \cite[Section 4]{KS}.

\medskip
\begin{proof}

Denote $D=D(A)=A-A$, $S=S(A)=A+A$. 
Let $P\subseteq D$ be a popular energy subset of $D$ by energy. Note that $P,\Delta$ throughout the paper have different meaning, and in this section these are not the quantities used in the proofs of Theorems \ref{thm:1} and \ref{thm:new_energy}, also different from one another. Namely, now $P$ is defined as follows. There exists  some $\Delta:\,\frac{\E(A)}{2|A|^2}\leq \Delta \leq |A|$, such that each $x\in P$ has approximately $\Delta$, that is between $\Delta$ and $2\Delta$ realisations and 
\begin{equation}\label{popen} \E(A) \sim \sum_{x\in P}r^2_{A-A}(x) \gtrsim  |P|\Delta^2.
\end{equation}
Such $P$ exists by the dyadic version of the  pigeonhole principle. We will not keep track of powers of $\log|A|$, for they go slightly out of hand in the proof of Lemma \ref{lemma:rect}.

Furthermore, define a plane point set
\begin{equation}\label{pa}
\mathcal P(A) \subseteq A\times A := \{(a,a'):\;a-a'\in P\},\; \mbox{so }\;|\mathcal P(A) | \approx  |P|\Delta.\end{equation}

\begin{lemma}
	\label{lemma:rect} 
	There exists $\tilde A\subseteq A$, with $|\tilde A|\gg|A|$,  a pair of subsets $A',A''\subseteq \tilde A$ and a natural number $q$ with the following property: $|A'|\gtrsim |A''|\geq q$, $|A'|\gtrsim |A|$, and for each $a\in A'$ there are at least $q\sim |\mathcal P(\tilde A)|/|A|$ points $(a',a'')\in \mathcal P(\tilde A) \cap (A'\times A'')$, where the point set $P(\tilde A)\subseteq \tilde A\times \tilde A$ is the popular energy set defined in terms of $\tilde A$ via  \eqref{popen}, \eqref{pa}. 
	
	Besides, $\E(\tilde A)\sim \E(A',A'').$ 
\end{lemma}

Informally, Lemma \ref{lemma:rect} claims that there exists a positive proportion (in fact, of density arbitrarily close to $1$ as the proof shows) proportion subset $\tilde A$ of $A$, such that much of the energy of $\tilde A$ (up to factors of powers of $\log|A|$) is supported on some ``regular'' point set $\mathcal P'$, such that $\mathcal P'$ gets covered by a {\em wide} rectangle:  $\mathcal P'\subseteq A'\times A'' \, \subseteq \, \tilde A\times \tilde A$, with, most importantly, the lower bound $|A'|\gtrsim |A|$, {\em and} (up to factors of powers of $\log|A|$) an expected lower bound of the number of  points of $\mathcal P'$ for every abscissa $a'\in A'$. The lower bound $|A'|\gtrsim |A|$ strengthens the above-mentioned earlier statement in \cite{RSS} and is crucial for the forthcoming argument. 
 
We remark that the claim or Lemma \ref{lemma:rect} remains true for any finite $\geq 1$ moment of the number-of-realisations function $r_{A-A},$  as well as its multiplicative analogue $r_{A/A}$.

 \medskip
We prove the lemma afterwards and now proceed with the proof of Theorem \ref{thm:sums}. Apply Lemma \ref{lemma:rect} and to ease on notation reset $\tilde A=A$ and $\mathcal P(\tilde A)=\mathcal P$.
Let $\mathcal P'=\mathcal P\cap(A'\times A'').$

For $\lambda\in A/A$ (nonzero, since $0\not \in A$) denote 
$$A'_\lambda = \{a\in A: \,\lambda a\in A'\},\;\mbox{ so }\;  \lambda A'_\lambda\;\subseteq\;A'.$$
Also denote, for brevity
\begin{equation}\label{ener}
\E^\times= \sum_\lambda |A'_\lambda|^2,
\end{equation}
so $\E^\times$ is the multiplicative energy of $A'$ and $A$. In the sequel sums in $\lambda$ mean sums over $ A/A$. Clearly,

\begin{equation}\label{ener1}
\sum_{\lambda} |A_\lambda| =  |A|^2,\qquad \sum_{\lambda} |A'_\lambda| =  |A||A'|.
\end{equation}

\medskip
Whether either $|A/A|=M|A|$ or $|AA|=M|A|,$ does not affect the argument, for $M$ only comes from estimate \eqref{b:1} in Lemma \ref{l:E_3} that we restate: \begin{equation}
\E_3(A'') \lesssim M^2|A||A''|^2,\qquad \E_3(A'_\lambda) = \E_3(\lambda^{-1} A'_\lambda) \lesssim M^2|A| |A'_\lambda|^2.
\label{one}
\end{equation}

\medskip
Let us estimate the number $\mathcal T_\Lambda$ of collinear point triples in  $S\times S \subset \R^2$ such that the the point triples are supported just on the lines with slopes in $\Lambda=A/A$. For the total number of collinear point triples $\mathcal T(S)\geq \mathcal T_\Lambda$ in $S\times S$  we have the unconditional upper bound, which is the standard implication of the Szemer\'edi-Trotter theorem:
\begin{equation}\label{ct} 
\mathcal T_\Lambda\leq \mathcal T(S) \lesssim |S|^4.
\end{equation}

\medskip
To get a lower bound on $\mathcal T_\Lambda$ we consider a line $y=\lambda x + d$, where some $d\in P$, $\lambda\in \Lambda$ and estimate the minimum number of points of $S\times S$, supported on this line.

This means, we are looking at some $a,b,a',b'\in A$ such that
$$
a+b - \lambda(a'+b') =  d.
$$
The latter equation is satisfied if one chooses $b=\lambda b'$, with any $b'\in A'_\lambda$: for any $a'\in A'_\lambda$ there are at least $q$ choices of $a\in A''$, where the number $q$ comes from Lemma \ref{lemma:rect}.

Hence, define a point set $\mathcal Q_\lambda\subseteq S\times S$ as follows:

\begin{equation}\label{qul}
\mathcal Q_\lambda:=\{ (a' + b, a + \lambda b): \,a\in A, \,a',b\in A'_\lambda, \mbox{ and }(a,\lambda a')\in \mathcal P'\}. 
\end{equation}
so for $(x,y)\in \mathcal Q_\lambda$ we have
$$
y-\lambda x = a - \lambda a',
$$
and since $\lambda a'\in A'$, for each pre-image $a'\in A'_\lambda$ there are at lest $q$ values of $a\in A''$,  by Lemma \ref{lemma:rect}, such that $a - \lambda a'\in P$.

The point set $\mathcal Q_\lambda \subset \mathbb R^2$ (by construction, see \eqref{qul}) is supported on at most $|P|$  parallel lines with the slope $\lambda$, making therefore $|\mathcal Q_\lambda|$ incidences with these lines. The number of collinear triples in the set $\mathcal Q_\lambda$ on these lines, giving the lower bound for the quantity $\mathcal T_\lambda$ is at least the uniform case (in other words, we use the H\"older inequality):
$$
\mathcal T_\lambda \geq  |P| (|\mathcal Q_\lambda|/|P|)^3 = |\mathcal Q_\lambda|^3/|P|^2.
$$

Comparing this with \eqref{ct} we obtain

\begin{equation}\label{int}
|S|^4|P|^2\gtrsim \sum_{\lambda\in \Lambda} |\mathcal Q_\lambda|^3.
\end{equation}

\medskip
For a fixed $\lambda$, by \eqref{qul} and Cauchy-Schwarz, one has
\begin{equation}
q|A'_\lambda|^2 \leq |\mathcal Q_\lambda|^{1/2} X^{1/2},
\label{cs}\end{equation}
where $X$ is the number of solutions of the following system of equations:
$$
a' + b= a''+ b', \qquad a + \lambda b = a''' +  \lambda b': \qquad a',a'',b,b'\in A'_\lambda,\;a,a''' \in A''.
$$
This can be rewritten as 
$$
\lambda(a' - a'') = \lambda(b'- b) = a-a'''.
$$
For $d\in A-A$, let $r''(d),r'_\lambda(d)$ denote the number of realisations of $d$ as a difference in $A''-A''$ and $A'_\lambda-A'_\lambda$, respectively. Then we have, tautologically by definition of $X$ and then H\"older inequality
$$
X \leq \sum_d r''(d)  r'_\lambda(d) r'_\lambda(d)  \leq \E^{1/3}_3(A'')   \E^{2/3}_3(A'_\lambda) ,
$$

We can now proceed with \eqref{cs} by  applying \eqref{one} to the above estimate for the quantity $X$, getting
$$
q |A'_\lambda|^2 \leq  |\mathcal Q_\lambda|^{1/2} (M^{1/3}|A|^{1/6}|A''|^{1/3}) (M^{2/3}|A|^{1/3}  |A'_\lambda|^{2/3}).
$$

\medskip
Before summing over $\lambda\in A/A$ let us rearrange as follows:
\begin{equation}
q|A''|^{-1/3}|A|^{-1/2}  |A'_\lambda|^{3} \leq  M  |\mathcal Q_\lambda|^{1/2}  |A'_\lambda|^{5/3},
\label{intm}\end{equation}

We now sum in $\lambda$. In the right-hand side of \eqref{intm} we apply the H\"older inequality:
 $$\sum_\lambda |\mathcal Q_\lambda|^{1/2}  |A'_\lambda|^{5/3}\leq \left(\sum_\lambda |\mathcal Q_\lambda|^3 \right)^{1/6} \left(\sum_\lambda |A'_\lambda|^2 \right)^{5/6}.$$

In the left-hand side of \eqref{intm}  use the Cauchy-Schwarz inequality and relations \eqref{ener}, \eqref{ener1}:
$$\E^\times=\sum_{\lambda} |A'_\lambda|^{1/2}|A'_\lambda|^{3/2} \leq \sqrt{ |A| |A'| } \sqrt{\sum_{\lambda}|A'_\lambda|^3}.$$

Applying  the standard Cauchy-Schwarz estimate $\E^\times\geq(|A||A|')^2/(M|A|)$ yields

$$
\sum_{\lambda}  |\mathcal Q_\lambda|^{3} \geq M^{-6} |A|^{-9} q^6 |A''|^{-2} |A'|^{-6} {\E^\times}^{7} \geq M^{-13} |A|^{-2} q^6 |A''|^{-2} |A'|^{8},
$$
Thus by \eqref{int} one has
$$
|S|^4\gtrsim M^{-13} |A|^{-2} q^6 |A''|^{-2} |A'|^{8} |P|^{-2}.
$$

Using Lemma
\ref{lemma:rect}
we have the worst possible case $q^6 |A''|^{-2} |A'|^{8} \gtrsim (|P|\Delta)^6$,
thus
$$
|S|^4\gtrsim M^{-13} |A|^{-2} |P|^4 \Delta^6 \gtrsim M^{-13} |A|^{-2} \E^4(A)\Delta^{-2}.$$
The claim of of Theorem \ref{thm:1} now follows from the upper bound
$$
\Delta \leq \E_3(A)/\E(A) \lesssim M^2|A|^3/\E(A),
$$
where for $\E_3(A)$ we use \eqref{b:1}, and the standard Cauchy-Schwarz lower bound $\E(A)\geq|A|^4/|S|.$ 
$\hfill\Box$
\end{proof}

\subsection{Proof of Lemma \ref{lemma:rect}}
\begin{proof}
The proof is a pigeonholing argument.

For brevity sake round up the values of all logarithms to integers.
In the notation of the proof of Theorem \ref{thm:sums} curtail $\mathcal P=\mathcal P(A)$. Partition $A$ in at most $\log|A|$ nonempty sets by popularity of abscissae in $\mathcal P$. That is for $i=1,\ldots,\log|A|$ each abscissa from the set $A_i$ supports between $2^{i-1}$ (inclusive) and $2^i$ (non-inclusive) points of $\mathcal P$. Let us set $q_i=2^i$  and further in the proof use $\approx q_i$ as a shortcut for a number between $2^{i-1}$ and $2^i$. For  $i=1,\ldots,\log|A|$ take $\mathcal P_i$, defined as the set of all points of $\mathcal P$ with abscissae in $A_i$ and dyadically partition the set of its ordinates by popularity as the union of sets $A_i^j$. 

As a result, $\mathcal P$ is covered by the union of at most $\log^2|A|$ disjoint rectangles $A_i\times A_i^j$. By a {\em rectangle} we mean Cartesian product. By symmetry we can transpose $i$ and $j$, so there is another cover symmetric with respect to the bisector $y=x$. 

By the pigeonhole principle, at least one half of the mass $|\mathcal P| \approx |P|\Delta$ lies in ``rich'' rectangles $A_i^j$,  each containing at least $\frac{1}{2 \log^2|A|} |P|\Delta$ points of $\mathcal P$. Rather than writing out powers of $2$ arising after such popularity arguments explicitly, we will often subsume them in the $\gg,\ll,\approx$ symbols.

There are two cases to consider.

{\sf Case 1.}  One of the rich rectangles has width or height, say $\gg |A|\log^{-10}|A|$. Then there is another, symmetric with respect to the $y=x$ bisector, and we are almost done, choosing the wider  of the two rectangles. 

Indeed, let $\mathcal R_i^j$ denote such a rich rectangle. We set $A''$ to be its projection on the $y$-axis.

Furthermore, $\mathcal R_i^j$ has base $A_i$, with $|A_i|\geq|A''|$ and  $|A_i|\gg |A|\log^{-10}|A|$ (by how the case has been defined).  For each $a\in A_i$ there are  $\approx q_i$ points of $\mathcal P$. In particular, there are  $\ll q_i$ points of $\mathcal P\cap \mathcal R_i^j$ with a given abscissa $a\in A_i$.  Clearly, $|A_i|q_i\ll |P|\Delta$ (the total number of points in $\mathcal P$) thus the maximum number of points of $\mathcal P$ in $\mathcal R_i^j$ with the same abscissa $a\in A_i$ is trivially 
\begin{equation}\label{inter}q_i \ll |P|\Delta/|A_i|.\end{equation} On the other hand, $\mathcal R_i^j$ is rich, that is contains $\gg \frac{1}{\log^2|A|} |P|\Delta$ points of $\mathcal P$.  We refine its base $A_i$ to the set $A'$ of rich abscissae, supporting  at least $q: = \frac{1}{16} \frac{1}{|A_i|}\frac{1}{\log^2|A|} |P|\Delta$ points of $P\cap \mathcal R_i^j$ each; by the pigeonhole principle these abscissae in $A'$ still support $\gg \frac{1}{\log^2|A|} |P|\Delta$ points of $\mathcal P\cap \mathcal R_i^j$.

Clearly, $q\leq |A''|$, the total number of ordinates in $\mathcal R_i^j$.

Since the maximum number of points of $\mathcal P$ per abscissa in $A_i$ is $q_i$, the minimum number of the latter rich abscissae is 
$$|A'|\gg \frac{1}{q_i}\frac{1}{\log^2|A|} |P|\Delta \gg \frac{1}{\log^2|A|}|A_i|,
$$
by \eqref{inter}.
Recall that by construction $|A_i|\geq |A''|$, as well as $|A_i|\gg |A|\log^{-10}|A|$, so $|A'|\gtrsim |A''|$ and $|A'|\gtrsim |A|$.
Thus, but for the last claim of the lemma about energy, to be dealt with in the very end of the proof, we are done with Case 1, having found the sets $A',A''$ and $\tilde{A} = A$. 

\medskip{\sf Case 2.} Let us show that Case 1 is the generic one, that is if there are no sufficiently wide or high rich rectangles apropos of $\mathcal P(A)$, we can refine $A$ to its (arbitrarily) high proportion subset $\tilde A$ and find ourselves in Case 1, relative to the point set $\mathcal P(\tilde A)\subseteq \mathcal P(A)$.

Indeed, suppose there are no sufficiently wide or high rich rectangles  $A_i^j$ in the above constructed covering of $\mathcal P(A)$ by rectangles, that is both projections of each rich rectangle are  $\leq |A|\log^{-10}|A|$ in size. Then we remove from $A$ the union  $A_i\cup \bigcup_j A_i^j$, calling the resulting set $A_1$. By construction, crudely, $|A_1|\gg (1- \log^{-5}|A|)|A|.$ 

On the other hand, at least half of the mass of $\mathcal P(A)$ has been removed. This means loss of at least a third of the energy. Indeed, $\mathcal P$ is supported on the union of $|\Delta|$ lines with the number of points of $A\times A$ on each line ranging between $N$ and $2N$ and now at least half of the point set $\mathcal P$ has been deleted. Let us estimate from below the difference $\E(A)-\E(A_1).$

To minimise the amount of energy lost  after the deletion of half of the point set $\mathcal P(A)$, supported on $P$ lines, each supporting between $\Delta$ and $2\Delta$ of $A\times A$) one should be deleting the poorest lines, one by one (stopping in the midst of a line is also allowed). To this end, the extremal case arises if half of the mass of $\mathcal P$ were supported on  lines with minimum occupancy $\Delta$ and the other half on lines with maximum occupancy $2\Delta$. In this extremal case one has $|\mathcal P|=4|P|\Delta/3$, the total energy $2|P|\Delta^2$ being supported on $|P|$ lines. Deleting the $2|P|/3$ poorest lines means being left with two-thirds of the energy. 

Thus $\E(A_1)\leq (1- \frac{1}{3} \log^{-1}|A|) \E(A)$.

So we pass from $A$ to $A_1$ and check if we are in Case 1 apropos of $\mathcal P(A_1)$. If yes, we are done (modulo the coming claim about energy), otherwise we iterate the Case 2 deletion procedure.  Once we have encountered Case 2, say  $\log^5|A|$ times, we still retain a  large subset, of $A$, containing a fraction of $A$, bounded from below as $\Omega\left((1- \log^{-5}|A|)^{\log^5|A|}\right)\gg 1$. The energy of this remaining large subset,  however, is at most 
$$(1- \frac{1}{3} \log^{-1}|A|)^{\log^5|A|} \E(A) = o(|A|^2),$$ which is a contradiction. Thus at some point throughout iteration we must encounter Case 1.

\medskip
To verify the final claim of the lemma about energy, it suffices to assume that the desired pair $A',A''$ has been found in Case 1 immediately, that is $\tilde A=A$. Clearly, $$\E(A',A'')\leq \E(A) \sim |P|\Delta^2.$$ On the other hand, the rectangle $\mathcal R=A'\times A''$ still contains $\sim |P|\Delta$ points of $P$, which are all supported on at most $|P|$ lines. 

Thus
$$\sum_{x \in P} r_{A'-A''} (x) \gg |P|\Delta. $$
It follows by Cauchy-Schwarz that
$$\E(A',A'') \gg \sum_{x \in P} r^2_{A'-A''} (x) \gg |P|\Delta^2 \sim \E(A). $$

$\hfill\Box$
 \end{proof}
 
 \section*{Acknowledgemnet} We thank O. Roche-Newton, J. Solymosi, S. Stevens and D. Zhelezov for their advise on exposition of the results in this paper.

\bigskip

\noindent{B.~Murphy\\
	School of Mathematics,\\
	University Walk, Bristol BS8 1TW, UK\\
	{\tt brendan.murphy@bristol.ac.uk}}
	
	\bigskip

\noindent{M.~Rudnev\\
	School of Mathematics,\\
	University Walk, Bristol BS8 1TW, UK\\
	{\tt misarudnev@gmail.com}}

\bigskip

\noindent{I.D.~Shkredov\\
	Steklov Mathematical Institute,\\
	ul. Gubkina, 8, Moscow, Russia, 119991}
\\
and
\\
IITP RAS,  \\
Bolshoy Karetny per. 19, Moscow, Russia, 127994\\
and 
\\
MIPT, \\ 
Institutskii per. 9, Dolgoprudnii, Russia, 141701\\
{\tt ilya.shkredov@gmail.com}

\bigskip

\noindent{Yu. N.~Shteinikov\\
	SRISA,\\
	Nahimovsky prosp. 36, building 1, Russia, 117218\\
	{\tt yuriisht@gmail.com}}


\begin{thebibliography}{99}

\bibitem{BW}
{\sc A.~Balog, T.D.~Wooley, }
\emph{A low--energy decomposition theorem, } Quart. J. Math.,  {\bf 68}:1 (2017), 207--226, doi: 10.1093/qmath/haw023.

	\bibitem{bur3} 
	{\sc J. Bourgain, S. Konyagin, I. Shparlinski,} 
	{\em Product sets of rationals, multiplicative translates of subgroups in residue rings and fixed points of the discrete logarithm, }
	  International Math. Research Notices  (2008), 1--29. 
	  
	 
	
	\bibitem{BCh} 
	{\sc J. Bourgain,  M-Ch. Chang,  }
	{\em On the size of k-fold sum and product sets of integers, } 
	J. Amer. Math. Soc. {\bf 17}:2 (2004), 473--497.
	
	

	
	\bibitem{ChS} 
	{\sc M-Ch. Chang, J. Solymosi,} 
	{\em Sum-product theorems and incidence geometry,}
J. Eur. Math. Soc. (JEMS) {\bf 9}:3 (2007),  545--560.

\bibitem{CG1} 
{\sc J. Cilleruelo, M. Garaev, }
{\em  Congruences involving product of intervals and sets with small multiplicative doubling modulo a prime and applications,} 
  Math. Proc. Cambridge Phil. Soc.,  Vol. 160, Issue 03, pp. 477--494, May 2016.
  
\bibitem{CG2}
{\sc J. Cilleruelo, M. Garaev,} 
{\em The congruence $x^x = \lambda \pmod p$,}
Proc. Amer. Math. Soc. (to appear).

\bibitem{E} {\sc G. Elekes,} {\em On the number of sums and products}, Acta Arith. {\bf 81} (1997), 365--367.

\bibitem{ER} {\sc G. Elekes, I. Z. Ruzsa, } {\em Few sums, many products,} Studia Sci. Math. Hungar. {\bf 40}:3 (2003),  301--308.




\bibitem{ES}
{\sc P.~Erd\H{o}s, E.~Szemer\'{e}di, }
\emph{On sums and products of integers, }
Studies in pure mathematics, 213--218, Birkh\"auser, Basel, 1983.

\bibitem{GS} {\sc A. Granville, J. Solymosi,} {\em Sum-product formulae,} in Recent Trends in Combinatorics, IMA vol. 159, pp 419--451, Springer 2016.


\bibitem{H-K}
{\sc D.~R.~Heath--Brown, S.~V.~Konyagin, }
{\em New bounds for Gauss sums derived from $k$th powers, and for Heilbronn's exponential sum, }
Quart. J. Math. {\bf 51} (2000), 221--235.





\bibitem{K_Tula} {\sc S.~V.~Konyagin,}
{\em Estimates for trigonometric sums and for Gaussian sums, }
IV International conference "Modern problems of number theory and its applications". Part 3 (2002), 86--114.


\bibitem{KS-}
{\sc S.V.~Konyagin, I.D.~Shkredov, }
{\em On sum sets of sets, having small product sets, }
Transactions of Steklov Mathematical Institute, {\bf 3}:290 (2015), 304--316.

	
	\bibitem{KS}{\sc S.V.~Konyagin, I.D.~Shkredov, }
{\em New results on sum--products in $\mathbb R$, }
 Proc. Steklov Inst. Math., {\bf 294}:78, (2016), 87--98. 
 
 
  \bibitem{KS1}
 {\sc S.~V.~Konyagin, I.~Shparlinski, }
 {\em Character sums with exponential functions, } Cambridge University Press, Cambridge, 1999.
 
 
 


\bibitem{57} {\sc B. Murphy, G. Petridis, O. Roche-Newton, M. Rudnev, I. D. Shkredov,} {\em New results on sum-product type growth over fields,} arXiv:1702.01003v3 [math.CO]  9 Mar 2017.


\bibitem{Li_R-N}
{\sc L. Li, O. Roche-Newton, }
{\em Convexity and a sum-product type estimate, } Acta Arith. {\bf 156}:3 (2012),  247--255.


\bibitem{MSS}
{\sc S.~Macourt, I.~D.~Shkredov, I.~Shparlinski, }
{\em Multiplicative energy of shifted sugroups and bounds on exponential sums with trinomials in finite fields, }
Canadian J. Math., accepted, arXiv:1701.06192v1 [math.NT] 22 Jan 2017.



\bibitem{Malykhin_p^2}
{\sc Yu.~V.~Malykhin, }
{\em Bounds for exponential sums over $p^2$, }
J. Math. Sci. {\bf 146}:2 (2007), 5686--5696.


\bibitem{Mit}
{\sc D.A. Mit'kin. }
{\em Estimation of the total number of total number  of the rational  points  on a set of curves in a simple finite field, }
Chebyshevsky sbornik {\bf 4}:4 (2003), 94--102.


\bibitem{VarII}
{\sc B. Murphy, O. Roche-Newton, I.D. Shkredov, }
\emph{Variations on the sum--product problem II, } arXiv:1703.09549v2 [math.CO] 29 Mar 2017. 


\bibitem{pap}
{\sc F. Pappalardi }
{\em On the Order of Finitely Generated Subgroups of Q*(modp) and Divisors of p-1.}
J. Number Theory
{\bf 57}:2 (1996),
207--222.

\bibitem{RRS} {\sc O. Roche-Newton, M. Rudnev, I. D. Shkredov,} {\em New sum-product type estimates over finite fields,} Adv. Math. {\bf 293}(2016), 589--605.

\bibitem{misha}
{\sc M.~Rudnev,}
{\em On the number of incidences between planes and points in three dimensions,}  Combinatorica (2017), 1--36, doi:10.1007/s00493-016-3329-6.

\bibitem{RSS} {\sc M. Rudnev, S. Stevens, I.D. Shkredov, }
{\em On The Energy Variant of the Sum-Product Conjecture,}  arXiv  1607.05053v5  [math.CO]  5 June 2017.


\bibitem{SS}  
{\sc T. Schoen and I. D. Shkredov, } {\em Additive properties of multiplicative subgroups of $\F_p$,} Quart. J. Math. {\bf 63}:3 (2012), 713--722. 

	\bibitem{SS1}  
{\sc T. Schoen, I. D. Shkredov, }
\emph{Higher moments of convolutions, }
J. Number Theory {\bf 133}:5 (2013), 1693--1737.

\bibitem{S_sp}
{\sc I. D. Shkredov, } {\em Some applications of W. Rudin's inequality to problems of combinatorial number theory,} Unif. Distrib. Theory {\bf 6}:2 (2011),  95--116.



\bibitem{S_ineq}
{\sc I. D. Shkredov, }
\emph{Some new inequalities in additive combinatorics, } MJCNT, {\bf 3}:2 (2013), 237--288.


\bibitem{s_mixed} {\sc I.~D.~Shkredov, }
\emph{Some new results on higher energies, }
Transactions of MMS, {\bf 74}:1 (2013), 35--73.

\bibitem{Ilmed}  
{\sc I. D. Shkredov, }
\emph{On exponential sums over multiplicative subgoups of medium size, }
Finite Fields Appl. {\bf 30} (2014), 72--87.




\bibitem{S_tripling}
{\sc I. D. Shkredov, }
\emph{On tripling constant of multiplicative subgroups, } 
Integers \#A75 {\bf 16} (2016), 1--9.


\bibitem{S_AA}
{\sc I. D. Shkredov, }
\emph{Some remarks on sets with small quotient set, }
Proceedings of the Steklov Institute of Mathematics, 
accepted; arXiv:1603.04948v2  [math.CO]  15 May 2016. 


\bibitem{SSV}
{\sc I. D. Shkredov, E. Solodkova, I. Vyugin, } 
\emph{On the additive energy of Heilbronn's subgroup, } 
Mat. Zametki, {\bf 101}:1 (2017), 43--57; English transl.  Mat. Zametki, {\bf 101}:1 (2017), 58--70.

\bibitem{SV}
{\sc I. D. Shkredov, I. Vyugin, }  {\em On additive shifts of multiplicative subgroups,} (Russian) Mat. Sb. {\bf 203} (2012), no. 6, 81--100; translation in Sb. Math. {\bf 203} (2012), no. 5-6, 844--863.


\bibitem{SZ}
{\sc I. D. Shkredov, D. Zhelezov, } 
\emph{On additive bases of sets with small product set, } International Math. Research Notices, accepted; arXiv:1606.02320v2 [math.NT] 14 Jun 2016. 

\bibitem{Iur}
{\sc Yu. Shteinikov,}
{\em Estimates of trigonometric sums over subgroups and some of their applications, }
Mathematical Notes {\bf 98}:3-4 (2015),  667--684.



\bibitem{So}
{\sc J. Solymosi, }
{\em Bounding multiplicative energy by the sumset, }
Adv. Math. {\bf 222}:2 (2009), 402--408.

\bibitem{SoT} {\sc J. Solymosi, G. Tardos,} {\em On the number of $k$-rich transformations,} Computational geometry (SCG'07), 227--231, ACM, New York, 2007.


\bibitem{SdZ} {\sc S. Stevens, F. de Zeeuw, } {\em An Improved Point-Line Incidence Bound Over Arbitrary Fields}, Bull. Lon. Math. Soc., accepted; arXiv:1609.06284v4  [math.CO] 20 September 2016.



\bibitem{sz-t}
{\sc E.~Szemer\' edi, W.~T.~Trotter,} 
{\em Extremal problems in discrete geometry,} 
Combinatorica {\bf 3} (1983), 381--392.


\bibitem{TV}
{\sc T. Tao, V. Vu, }
{\em Additive Combinatorics, }
Cambridge University Press (2006).

  
	

\end{thebibliography}
\end{document}